\documentclass[11pt]{article}
\usepackage{multicol}
\usepackage{amssymb,amsmath,amsthm,bm}
\usepackage[colorlinks=true, urlcolor=blue,linkcolor=blue, citecolor=blue]{hyperref}
\usepackage{mathrsfs,verbatim}
\usepackage{caption,cleveref }
\usepackage{graphicx, float}
\usepackage{color, mathtools, tikz, xcolor}
\usepackage{enumerate,enumitem}
\usepackage{tikz}
\usetikzlibrary{shapes.geometric}
\usepackage{pdfpages}
\usepackage{bbm}
\usepackage{algorithm}
\usepackage{algorithmic}
\usetikzlibrary{calc}

\usepackage[mathscr]{euscript}
\usepackage{appendix}
\usepackage{geometry}
\usepackage{lineno}
\usepackage[normalem]{ulem}
\newcommand{\eps}{\varepsilon}
\usetikzlibrary{shapes.geometric, arrows.meta}
\numberwithin{equation}{section}

\renewcommand{\l}{\left}
\renewcommand{\r}{\right}

\newcommand{\YC}[1]{\textcolor{blue}{#1}}
\usepackage{svg}
\usepackage{multicol}
\usepackage{amssymb,amsmath,amsthm,bm}
\usepackage[colorlinks=true, urlcolor=blue,linkcolor=blue, citecolor=blue]{hyperref}

\numberwithin{equation}{section}

	\newtheorem{theorem}{Theorem}[section]
	
        \newtheorem{question}[theorem]{Question}
	
	\newtheorem{coro}[theorem]{Corollary}
	
	\newtheorem{claim}[theorem]{Claim}
	\newtheorem{proposition}[theorem]{Proposition}
	\newtheorem{lemma}[theorem]{Lemma}

        \theoremstyle{definition}
	\newtheorem{defn}[theorem]{Definition}
    \newcommand{\card}[1]{|#1|}

\newenvironment{proofclaim}[1][Proof of claim]{\begin{proof}[#1]}{\end{proof}}

\setlist{nolistsep}

\title{Robustness for expander graphs}
\author{Yaobin Chen\thanks{Shanghai Center for Mathematical Sciences,~Fudan University,~Shanghai,~200438,~China.~{\tt ybchen21@m.fudan.edu.cn}.
Supported by National Natural Science Foundation of China grant 123B2012.}
\and Yu Chen\thanks{School of Mathematics and Statistics,~Beijing Institute of Technology,~Beijing,~100081,~China.~{\tt yu.chen2023@bit.edu.cn} (Y. Chen), {\tt han.jie@bit.edu.cn} (J. Han), {\tt jingwen.zhao@bit.edu.cn} (J. Zhao). 
J. Han is supported by the National Natural Science Foundation of China (12371341).}
\and Jie Han\footnotemark[2] 
\and Jingwen Zhao\footnotemark[2]}
\date{\today}

\begin{document}

\maketitle
\begin{abstract}
We study robust versions of properties of $(n,d,\lambda)$-graphs, namely, the property of a random sparsification of an $(n,d,\lambda)$-graph, where each edge is retained with probability $p$ independently. 
We prove such results for the containment problem of perfect matchings, Hamiltonian cycles, and triangle factors. 
These results address a series of problems posed by Frieze and Krivelevich.

First we prove that given $\gamma>0$, for sufficient large $n$, any $(n,d,\lambda)$-graph $G$ with $\lambda=o(d)$, $d=\Omega(\log n)$ and $p\ge\frac{(1+\gamma)\log n}{d}$, $G\cap G(n,p)$ contains a Hamiltonian cycle (and thus a perfect matching if $n$ is even) with high probability. This result is asymptotically optimal. 

Moreover, we show that for sufficient large $n$, any $(n,d,\lambda)$-graph $G$ with $\lambda=o(\frac{d^2}{n})$, $d=\Omega(n^{\frac{5}{6}}\log^{\frac{1}{2}}n)$ and $p\gg d^{-1}n^{\frac{1}{3}}\log^{\frac{1}{3}} n$, $G\cap G(n,p)$ contains a triangle factor with high probability. Here, the restrictions on $p$ and $\lambda$ are asymptotically optimal.

Our proof for the triangle factor problem uses the iterative absorption approach to build a spread measure on the triangle factors, and we also prove and use a coupling result for triangles in the random subgraph of an expander $G$ and the hyperedges in the random subgraph of the triangle-hypergraph of $G$.
\end{abstract}

\section{Introduction}
A classical problem in combinatorics is under which conditions a given graph contains a specific spanning structure. For example, Dirac~\cite{dirac1952some} proved that if an $n$-vertex graph $G$ has a minimum vertex degree at least $n/2$, then $G$ contains a Hamiltonian cycle. Over the past decades, the study of minimum degree conditions for other spanning substructures has grown into an essential branch of combinatorics.

A related field in combinatorics concerns the study of spanning structures in random graphs and pseudorandom graphs. For example, Pos{\'a} \cite{PosaHam} and Korshunov \cite{Korshunovham} independently showed that $G(n,p)$ contains a Hamiltonian cycle with high probability if $p\gg \log n /n$. Prominent examples~\cite{erdHos1968random,alon1993threshold,alon1992spanning,bollobas1985matchings,johansson2008factors,montgomery2019spanning,riordan2000spanning,de1979long} include the thresholds of occurrence for the containment of a perfect matching, the containment of a clique factor, etc., in random graphs.

Following the fruitful study of random graphs, it is natural to explore families of deterministic
graphs that behave in a certain sense like random graphs; these are called pseudorandom
graphs. 
One special class of pseudorandom graphs that has been studied extensively is the class of
spectral expander graphs, also known as $(n, d, \lambda)$-\emph{graphs}. An $(n,d,\lambda)$-\emph{graph} is an \(n\)-vertex \(d\)-regular graph whose second largest absolute value eigenvalue is at most $\lambda$. The \emph{expander mixing lemma} shows that $\lambda$ governs the edge distribution of $G$. The smaller $\lambda$ means that the edge distribution of $G$ more resembles that of $G(n,d/n)$. It is convenient to quantify this in terms of $(q,\beta)$-\emph{bijumbledness}: a graph $G$ with $n$ vertices is $(q,\beta)$-bijumbled for some $q\in [0,1]$ and $\beta>0$ if for every $X,Y\subseteq V(G)$, we have 
\begin{align*}
    |e(X,Y)-q|X||Y||\le \beta\sqrt{|X||Y|},
\end{align*}
where $e(X,Y)$ denotes the number of pairs $(u,v)\in X\times Y$ with $uv\in E(G)$ and note that under this definition, the edges in $X\cap Y$ are counted twice.
The expander mixing lemma shows that if $G$ is an $(n,d,\lambda)$-graph, then it is $(d/n,\lambda)$-bijumbled.

Recently, a breakthrough result of Dragani{\'c}, Montgomery, Correia, Pokrovskiy, and Sudakov~\cite{draganic2024hamiltonicity}, proved that an expander graph, and therefore an $(n,d,\lambda)$-graph with $\lambda\ll d$, contains a Hamiltonian cycle. For a graph $G$ and a vertex subset $X$, we denote the neighbor set of $X$ in $V(G)\backslash X$ by $N_{G}(X)$ (when $G$ is clear from the context, we use $N(X)$). We say that a graph $G=(V, E)$ is a $C$-\emph{expander} if for every vertex set $X$ with $1\le |X|\le \tfrac{n}{2C}$, $|N(X)|\ge C|X|$, and for any two disjoint sets with size at least $n/2C$, there exists an edge between them. 
\begin{theorem}[\cite{draganic2024hamiltonicity}]\label{expander graph Hamiltonian}
    For every sufficiently large \(C>0\). Let $G$ be a $C$-expander graph. Then $G$ contains a Hamiltonian cycle. In particular, there exists a constant $\varepsilon>0$ such that, if $G$ is an $(n,d,\lambda)$-graph with $\lambda\le \varepsilon d$, then $G$ contains a Hamiltonian cycle. 
\end{theorem}


It is natural to ask whether the random model and pseudorandom model can be combined. One interpretation of probabilistic threshold, initially suggested by Krivelevich, Lee, and Sudakov~\cite{KLS}, is as a measure of \emph{robustness}. 
They show that for an $n$-vertex graph $G$ with $\delta(G)\ge n/2$ and $p=\Omega(\log n/n)$, then the random sparsification $G_p$, obtained by keeping each edge of $G$ independently with probability $p$, contains a Hamiltonian cycle with high probability. 
Under dense graph or hypergraph settings, the robustness problems for perfect matchings, spanning trees, and Hamiltonian cycles have been widely studied in the past decades \cite{joos2023robust,pham2022toolkit,allen2024robust,han2025transversal,bastide2024random,han2025rainbow,chen2024thresholds}. 
Instead of dense graphs, this paper focuses on the robustness of sparse expander graphs. For an introduction to the latter, we direct the reader to the excellent survey by Krivelevich and Sudakov~\cite{Krivelevich-Sudakov2006}.

In 2002, Frieze and Krivelevich~\cite{Friezerandomhamsparse} initiated the study of robustness in sparse graphs with Hamiltonicity. They obtained the following result.
\begin{theorem}[\cite{Friezerandomhamsparse}]
    Let $G$ be an $(n,d,\lambda)$-graph with $\lambda=o(\frac{d^{5/2}}{n^{3/2}\log^{3/2} n})$. Then for any function $\omega(n)$ tending to infinity arbitrarily slowly:
    \begin{enumerate}[label=$(\roman*)$]
        \item if $p(n)=\frac{1}{d}\l(\log n +\log \log n -\omega(n)\r)$, then $G_p$ contains no Hamiltonian cycles with high probability.
        \item if $p(n)=\frac{1}{d}\l(\log n +\log \log n +\omega(n)\r)$, then $G_p$ contains a Hamiltonian cycle with high probability.
    \end{enumerate}
\end{theorem}
    
 In their paper, Frieze and Krivelevich conjectured that when the degree is linear $n$, the weakest possible condition can replace the restriction of $\lambda$, that is, $\lambda=o(d)$, and also posed a question about the existence of a perfect matching in the random subgraphs. 
 Our first result settles this question and provides the threshold for perfect matching and the Hamiltonian cycle by demonstrating the robustness of the expander property for pseudorandom graphs with degree $\Omega(\log n)$. 
  \begin{theorem}\label{thm:main_thm_ndlambda}
    Given a constant $\gamma\in(0,1]$. Let $G$ be an $(n,d,\lambda)$-graph with $\lambda=o(d)$ and $d=\Omega(\log n)$. Then the following holds. 
    \begin{enumerate}[label=$(\roman*)$]
    \item  If $p(n)=(1+\gamma)\frac{\log n}{d}$, then $G_p$ contains a Hamiltonian cycle with high probability.
    \item If $p(n)=(1-\gamma)\frac{\log n}{d}$, then $G_p$ contains an isolated vertex and so no perfect matchings nor Hamiltonian cycles with high probability.
    \end{enumerate}
 \end{theorem}
 

As we only rely on the edge distribution of $G$, we extend~\cref{thm:main_thm_ndlambda} to bijumbled graphs. 

 \begin{theorem}\label{thm:main_thm_bijumbled}
    Given constants $\gamma,\alpha \in(0,1]$. Let $n\in\mathbb{N}$. Assume that $G$ is a $(q,\beta)$-bijumbled graph on vertex set $[n]$ with $\beta=o(qn)$, $q=\Omega(\log n/n)$ and $\delta(G)\ge \alpha qn$. Then the following holds for all sufficiently large $n$. 
    \begin{enumerate}[label=$(\roman*)$]
    \item  If $p(n)=(1+\gamma)\frac{\log n}{\alpha qn}$, then $G_p$ contains a Hamiltonian cycle with high probability and so contains a perfect matching with high probability when $n$ is even.
    \item If \(p(n)=\frac{\log n}{4qn}\), then $G_p$ contains an isolated vertex and has no perfect matchings nor Hamiltonian cycles with high probability.
    \end{enumerate}
 \end{theorem}


Note that in \Cref{thm:main_thm_bijumbled} we need to impose a minimum degree as \(\delta(G)\ge \alpha qn\), where $\alpha$ can be arbitrarily small.

The 1-statement of 
\Cref{thm:main_thm_bijumbled} (and also \Cref{thm:main_thm_ndlambda}) is proved by showing that with high probability, the random sparsification $G_p$ is a $C$-expander for sufficiently large $C$.
Then $G_p$ is Hamiltonian by~\Cref{expander graph Hamiltonian}.
We note that the proof of~\Cref{expander graph Hamiltonian} from~\cite{draganic2024hamiltonicity} is highly non-trivial and technical.
So for perfect matching, we provide a self-contained, simple proof (Lemma~\ref{lem:tutte–expander_lemma}) by showing that all $ C$-expanders of even order for $C\ge 3$ have a perfect matching, by verifying \textit{Tutte's condition}.
Moreover, we believe that a bipartite version of this result could be beneficial, and thus we include it here.
It can be proved (simply) by verifying Hall's condition.

We say a bipartite graph $G$ with parts $A$ and $B$ \emph{\((q,\beta)\)-bijumbled} if $|A|=|B|$ and for every $X\subseteq A$ and $Y\subseteq B$, we  have
\[|e(X,Y)-p|X||Y||\le \beta \sqrt{|X||Y|}.\]

\begin{theorem}
    Given constants $\gamma,\alpha\in(0,1]$. Let $n\in\mathbb{N}$ be an integer. Assume that $G$ is a balanced $(q,\beta)$-bijumbled bipartite graph on vertex set $[2n]$ with $\beta\le\varepsilon qn$, $q=\Omega(\log n/n)$ and $\delta(G)\ge \alpha qn$. Then the following holds for all sufficiently large $n$. 
    \begin{enumerate}[label=$(\roman*)$]
    \item  If $p(n)=(1+\gamma)\frac{\log n}{\alpha qn}$, then $G_p$ contains a perfect matching with high probability.
    \item If \(p(n)=\frac{\log n}{4 qn}\), then $G_p$ contains an isolated vertex and has no perfect matchings with high probability.
    \end{enumerate}
 \end{theorem}

\medskip
In the same paper, Frieze and Krivelevich~\cite{Friezerandomhamsparse} posed the robustness problem about the triangle factor. 
Here, we determine the robustness threshold for the existence of a triangle factor in the random subgraph of $(n,d,\lambda)$-graphs, for sufficiently dense expander graphs. 
\begin{theorem}\label{thm:main_thm_triangle_factor}
Let $n\in \mathbb{N}$ be an integer with $3|n$ and $0<1/n\ll1/C\ll\varepsilon\ll1$. Let $G$ be an $(n,d,\lambda)$-graph with $d\ge Cn^{\frac{5}{6}}\log^{\frac{1}{2}}n$ and $\lambda\le\varepsilon d^2/n$. 
\begin{enumerate}[label=$(\roman*)$]
\item\label{traingle factor1}    If $p(n)\gg \frac{n^{\frac{1}{3}}\log^{\frac{1}{3}} n}{d}$, then $G_p$ contains a triangle factor with high probability.
\item\label{traingle factor2} If $p(n)\ll \frac{n^{\frac{1}{3}}\log^{\frac{1}{3}}n}{d}$, then $G_p$ contains no triangle factors with high probability. 
\end{enumerate}

\end{theorem}

We remark that in this result, the restriction of $\lambda$ is optimal due to an excellent construction of Alon~\cite{Alon_construction}, which shows that there exists a triangle-free $(n,d,\lambda)$-graph with $\lambda=\Omega(d^2/n)$. Since we need to construct a much denser structure than a perfect matching, our strategy to examine the expander property for the random model fails in this case. 

Our proof of~\Cref{thm:main_thm_triangle_factor} is based on the recent breakthrough in \cite{FKNP}, which reduces the problem to specifying a ``spread'’ distribution on the desired guest structure in the (deterministic) host structure. We apply the iterative absorption technique to construct such a spread distribution. 
Since the host graph is sparse, we will prove and utilize a sparse version of the coupling lemma (known as a nice result of Riordan~\cite{riordan2022random} for the dense case) and employ a novel algorithm to complete the "cover-down" step. 
In addition, we will use a recent lemma of Ferber, Han, Mao, and Vershynin~\cite{ferber2024hamiltonicity} on random induced subgraphs of $(n,d,\lambda)$-graphs, which guarantees the expander property for random induced subgraphs. 

\medskip
\noindent
\textbf{Organization.}
In Section~\ref{pre}, we will introduce some notations and probabilistic tools. In Section~\ref{robustexpan}, we prove the robustness of the expander property. In Section~\ref{perfecmatham}, we will show our main results, Theorem~\ref{thm:main_thm_ndlambda}. In Section~\ref{sparse couple}, we show a coupling lemma for triangles in sparse graphs. In Section~\ref{iterative ab}, we prove~Theorem~\ref{thm:main_thm_triangle_factor}. At last, we conclude our paper in Section~\ref{concluding}.

\section{Preliminaries}~\label{pre}
In this section, we first introduce some notations and probabilistic tools. Given a graph $G=(V,E)$ and a subgraph $H\subseteq G$. 
We denote the number of edges incident to $v$ in $H$ by $d_H(v)$. 
For a subset $X$ of $V(G)$, let $d_H(v, X)$ be the number of edges between vertex $v$ and $X$. 
We say a graph is $(n,(1\pm\gamma)d,\lambda)$-\emph{graph}, if every vertex has the degree in the interval $(1\pm\gamma)d$ and the second largest eigenvalue in absolute value is at most~$\lambda$. 
Now, we list some useful lemmas used in the paper.
\begin{lemma}[\cite{chernoff1952measure}, Chernoff's bound]\label{lem:chernoff_bound}
Let $X$ be either:
\begin{enumerate}
        \item[$\bullet$] a sum of independent random variables, each of which takes values in $[0,1]$, or
        \item[$\bullet$] hypergeometrically distributed (with any parameters).
\end{enumerate}
Then for any $\delta>0$ we have \[\mathbb{P}\left[X\le(1-\delta)\mathbb{E}[X]\right]\le{\exp\left(-\delta^2\mathbb{E}[X]/2\right)},\text{ and } \mathbb{P}\left[X\ge(1+\delta)\mathbb{E}[X]\right]\le{\exp\left(-\delta^2\mathbb{E}[X]/(2+\delta)\right)}.\]
\end{lemma}

The next well-known lemma describes the ``random-like'' behavior of $(n,d,\lambda)$-graphs. 
\begin{lemma}[\cite{alon2016probabilistic}, Expander mixing lemma]\label{expandermixing}
    Let $G$ be an $(n,d,\lambda)$-graph. Then, for any two subsets $S,T\subseteq V(G)$, we have
    \[\left|e(S,T)-\frac{d}{n}|S||T|\right|\le \lambda\sqrt{|S||T|}.\]
\end{lemma}

In~\cite{ferber2024hamiltonicity}, Ferber, Han, Mao, and Vershynin generalized the expander mixing lemma to almost regular expander graphs as follows.
\begin{lemma}[\cite{ferber2024hamiltonicity}, mixing lemma for almost regular expanders]\label{lem:mixing_lemma_for_almost_expnaders}
Let $G$ be an $(n,(1\pm\gamma)d,\lambda)$-graph. Then, for any two subsets $S,T\subseteq V(G)$, we have
\[\frac{(1-\gamma)^2d|S||T|}{(1+\gamma)n}-\eps\le e(S,T)\le \frac{(1+\gamma)^2d|S||T|}{(1-\gamma)n}+\eps, \]
where $\eps=\frac{1+\gamma}{1-\gamma}\lambda\sqrt{|S||T|}$.
\end{lemma}

The next theorem will help us maintain the expander property for a random subset when constructing the spreadness measure.
\begin{theorem}[\cite{ferber2024hamiltonicity}, random subgraphs of spectral expanders]\label{thm:random_subgraphs_of_spectral_expanders}
Let \(\gamma\in(0,1/200]\) be a constant. 
There exists an absolute constant $C=C_{\ref{thm:random_subgraphs_of_spectral_expanders}}>0$ such that the following holds for all sufficiently large $n$. 
Let \(d,\lambda>0\), let \(\sigma\in [1/n,1)\), and let $G$ be an $(n,(1\pm\gamma)d,\lambda)$-graph. 
Let $X\subseteq V(G)$ with $|X|=\sigma n$ be a subset chosen uniformly at random, and let $H:=G[X]$ be the subgraph of $G$ induced by $X$. 
Assume that \[\sigma d>C\gamma^{-2}\log n\text{ and } \sigma\lambda>\sqrt{\sigma d\log n}.\]
Then with probability at least $1-n^{-1/6}$, $H$ is a $(\sigma n,(1\pm2\gamma)\sigma d,6\sigma\lambda)$-graph.
\end{theorem}
Recently, Morris~\cite{morris2025clique} proved the following result about the existence of clique factors in pseudorandom graphs. This will play an essential role in our proof of Theorem~\ref{thm:main_thm_triangle_factor}.

\begin{theorem}[\cite{morris2025clique}]\label{thm:bijumbled_graph}
    For every \(3\le r\in \mathbb{N}\) and \(c>0\) there exists an \(\eta>0\) such that any \(n\)-vertex \(\l(p,\beta\r)\)-bijumbled graph \(G\) with \(n\in r\mathbb{N}\), \(p>0\), \(\delta(G)\ge cpn\) and \(\beta\le \eta p^{r-1} n\), contains a \(K_r\)-factor.
\end{theorem}

\section{Robust expander}\label{robustexpan}
As discussed in the introduction, the proof of \Cref{thm:main_thm_bijumbled} reduces to verifying the robustness of the expander property. 
In particular, it suffices to show that the random sparsification \(G_p\) of a \((q,\beta)\)-bijumbled graph \(G\) is a \(C\)-expander with high probability, as stated in the following theorem.
Recall that a graph $G$ is called a $C$-expander if,
\begin{enumerate}[label=(\roman*)]
    \item $|N(X)|\ge C|X|$ for every $X\subseteq V(G)$ with $|X|<n/C$;
    \item\label{the-jointness} there is an edge between every two disjoint vertex sets of size at least $n/2C$.
\end{enumerate}

\begin{theorem}\label{thm:robust_expander}
    Let \(1/n\ll\varepsilon\ll1/C\ll \delta\ll \gamma, \alpha \le 1\). 
    Let $G$ be an \(n\)-vertex $(q,\beta)$-bijumbled graph with $\beta\le \eps qn$ and $\delta(G)\ge \alpha qn$. 
    If $p(n)=\l(1+\gamma\r)\frac{\log n}{\alpha qn}$, then $G_p$ is a \(C\)-expander with high probability.
\end{theorem}

\begin{proof}
    Let $H:=G_p$ with $p=(1+\gamma)\log n/(\alpha qn)$. 
    We first show that with probability $1-o(1)$, every vertex has degree at least $\delta \log n$.  Let $v\in V(G)$, the probability that $\deg_H(v)< \delta\log n$ is at most 
    \begin{align*}
        \mathbb{P}[\deg_H(v)<\delta \log n]&\le \sum_{i=0}^{\delta\log n}\binom{\deg_G(v)}{i}p^i(1-p)^{\deg_G(v)-i}\\
        &\le \sum_{i=1}^{\delta \log n}\left(\frac{e \deg_G(v)}{i}\right)^i p^i(1-p)^{\deg_G(v)-i}+(1-p)^{\deg_G(v)}\\
        &\le \log n\l(\frac{2e}{\delta}\r)^{\delta\log n}\exp(-(1+\frac{\gamma}{2})\log n)=o(n^{-1-\frac{\gamma}{4}}).
    \end{align*}
    By the union bound, there exists a vertex with degree at most $\delta \log n$ in $H$ with probability at most $o(1)$.

    Next, we will show that $H$ is a $C$-expander with high probability.

    We verify the second condition first. In fact, we verify a stronger condition, which will help us prove the expansion condition later. For any two sets $A_1,A_2$ with $|A_1|=|A_2|=\frac{n}{2C^2}$, by the bijumbledness,
    \[e_G(A_1,A_2)\ge q|A_1||A_2|-\beta\sqrt{|A_1||A_2|}\ge \frac{qn^2}{8C^4}.
    \]
    Since each edge is included in $H$ with probability $p$ independently, we have that $\mathbb{E}(e_H(A_1,A_2))=pe_G(A_1,A_2)\ge \frac{n\log n}{8\alpha C^4}$. So by Lemma~\ref{lem:chernoff_bound}, we have
    \begin{align*}
        \mathbb{P}\left[e_H(A_1,A_2)\le (1-\varepsilon)pe_G(A_1,A_2)\right]\le \exp\left(-\frac{\eps^2n\log n}{16\alpha C^4} \right).
    \end{align*}
    Thus, the probability that the second condition fails is at most 
    \[\binom{n}{\frac{n}{2C^2}}\binom{n}{\frac{n}{2C^2}}\exp\left(-\frac{\eps^2 n\log n}{16\alpha C^4}\right)=o(1).\]

    To see the first condition, we divide the proof into four cases. By the discussion above, we can assume that in $H$, every vertex has degree at least $\delta \log n$.

    For $|X|\le \tfrac{\delta \log n}{2C}$, choose an arbitrary vertex $v\in X$. Observe that $|N_H(X)|\ge d_H(v,V(G)\setminus X)\ge C|X|$, as desired. 

    For all $\tfrac{\delta\log n}{2C}\le |X|\le \tfrac{n}{\log n}$, we want to show that $|N_H(X)|\ge C|X|$ with high probability. It suffices to prove that $e_H(Y)\le \tfrac{ \delta\log n}{2C}|Y|$ with high probability for all $\tfrac{\delta \log n}{2C}\le |Y|\le \tfrac{(C+1)n}{ \log n}$. Indeed, since $e_H(X\cup N_H(X))\ge \delta \log n|X|-\frac{\delta\log n|X|}{2C}$, which immediately implies that $|N_H(X)|\ge C|X|$ for every $\tfrac{\delta \log n}{2C}\le |X|\le \frac{n}{\log n}$. 
    By the bijumbledness, we have that 
    \begin{align}
        e_G(Y)\le q|Y|^2+\beta |Y|\le 2\varepsilon qn|Y|.
    \end{align}
    Thus, by the property of the Bernoulli distribution, we have 
    \begin{align*}
   &\binom{n}{|Y|}\mathbb{P}\left[\text{Bin}\left(e_G(Y),(1+\gamma)\frac{\log n}{\alpha qn}\right) > \frac{\delta\log n}{2C}|Y|\right]\\\le&\binom{n}{|Y|} \mathbb{P}\left[\text{Bin}\left(2\varepsilon qn|Y|,(1+\gamma)\frac{\log n}{\alpha qn}\right)> \frac{ \delta\log n}{2C}|Y|\right]\\
    \le& \l(\frac{en}{|Y|}\r)^{|Y|}\cdot\l(\frac{2e\varepsilon qn|Y|\tfrac{2\log n}{\alpha qn}}{\tfrac{ \delta\log n}{2C}|Y|}\r)^{\frac{ \delta\log n}{2C}|Y|}=\l(\frac{en}{|Y|}\cdot (8e\varepsilon C/\delta\alpha)^{\frac{\delta\log n}{2C}}\r)^{|Y|}
    \le \l(\frac{1}{2}\r)^{|Y|}.
    \end{align*}
    The last inequality holds as $\varepsilon\ll 1/C\ll\delta\ll  \alpha.$
    Sum over the size of $Y$ from $\frac{\delta \log n}{2C}$ to $\frac{(C+1)n}{ \log n}$, we see that $e_H(Y)\le \frac{\log n}{2C}|Y|$ with high probability.

 For $\frac{n}{\log n}\le |X|\le \frac{n}{2C^2}$, for the contrary, assume $|N_H(X)|< C|X|$. Write \(s:=|X|\). Let $Y\subseteq V(G)\setminus X$ with size $Cs$ such that $N_H(X)\subseteq Y$. Then, by the bijumbledness and $\varepsilon\ll 1/C$, we have that 
\begin{align*}
    e_G(X,X\cup Y)&\le q|X||X\cup Y|+\beta\sqrt{|X||X\cup Y|}
    \\&\le q(C+1)s^2+\beta\sqrt{C+1}s\\
    &\le \l(\frac{1}{C}qn+\beta\sqrt{C}\r)s\le \l(\frac{1}{C}+\varepsilon\sqrt{C}\r)qns\le \frac{3}{2C}qns.
\end{align*}
So $\mathbb{E}[e_H(X,X\cup Y)]=pe_G(X,X\cup Y)\le \frac{3\log n}{2\alpha C} \cdot s $. On the other hand, by our assumption and $N_H(X)\subseteq Y$, we have that 
 \begin{align*}
    e_H(X,X\cup Y)=e_H(X,V(G))=\sum_{v\in X}d_H(v)
    \ge \delta\log n \cdot s\ge \frac{2}{\alpha C}\log n\cdot s.
\end{align*}
The last inequlaity holds as $\frac{1}{C}\ll \delta\ll \alpha$. 
By Lemma~\ref{lem:chernoff_bound}, we have 
\begin{align*}
    \mathbb{P}\left[N_H(X)\subseteq Y\right]\le \mathbb{P}\left[e_H(X,X\cup Y)\ge \frac{2}{\alpha C}\log n\cdot s\right]\le\exp\left(-\frac{ s\log n }{1000\alpha C}\right).
\end{align*}
By summing over all such $X$ and $Y$, the probability that there exists a vertex set $X$ with $|N_H(X)|\le Cs$ is at most

 \begin{align*}
    \sum_{s=\frac{n}{\log n}}^{\frac{n}{2C^2}}\binom{n}{Cs}\binom{n}{s}\exp(-\frac{  s\log n }{1000C\alpha})
    &\le \sum_{s=\frac{n}{\log n}}^{\frac{n}{2C^2}}\l(\frac{en}{Cs}\r)^{Cs}\cdot \l(\frac{en}{s}\r)^s\exp(-\frac{s}{C}\log n )\\
    &\le \sum_{s=\frac{n}{\log n}}^{\frac{n}{2C^2}}\exp(2Cs\log\log n-\frac{s}{C}\log n)=o(1/n).
\end{align*}
Hence, with high probability that for any $|X|\le \frac{n}{2C^2}$, $|N_H(X)|\ge C|X|$ holds.

Finally, for $\frac{n}{2C^2}\le |X|\le \frac{n}{2C}$, by the proof of property~\ref{the-jointness}, we obtain that $|N_H(X)|\ge n-|X|-\frac{n}{2C^2}\ge C|X|$. Therefore, $H$ is a $C$-expander with high probability.
\end{proof}

\section{Perfect Matchings and Hamiltonian cycles from $C$-expander}\label{perfecmatham}
In this section, we will prove \Cref{thm:main_thm_ndlambda} and \ref{thm:main_thm_bijumbled}. We begin with the following lemma.

\begin{lemma}\label{lem:tutte–expander_lemma}
    If $G$ is a $C$-expander graph for some $C\ge 3$ with $|V(G)|$ even, then $G$ contains a perfect matching. 
\end{lemma}
\begin{proof}

    We will verify that \(G\) satisfies Tutte's condition, i.e., for every subset \(S\subseteq V(G)\), the number of odd connected components of \(G-S\) does not exceed \(\card{S}\). 
    Let \(A_1,\dots, A_m\) be all the maximal odd components of \(G-S\), listed in increasing order of size. 
    Since \(G\) is a \(C\)-expander, it is connected. 
    Assume that \(S\) is nonempty and \(m\ge 2\) (It is trivial for $m=1$). 
    By the maximality of each component, we have \(N_G(A_i)\subseteq S\) for all \(i\in [m]\). 

    We first consider the case \(|A_m|\ge n/2C\). 
    Note that we must have \(|\bigcup_{i\in [m-1]}A_i|<n/2C\), which implies \[m\le C(m-1)\le C|\cup_{i\in [m-1]}A_i|\le |\cup_{i\in [m-1]}N_G(A_i)|\le |S|.\] 
    Next, we consider the case \(|A_m|< n/2C\). 
    Let \(F_0\) be the union of a subfamily of the sets \(A_1,\dots, A_m\) such that $|F_0|$ is the largest under the restriction that \(|F_0|<n/2C\) (that is, adding to $F_0$ any set $A_j$ not contained in $F_0$ results a set of size larger than $n/2C$).
    Then let $F_1$ be found under the same rule in the remaining sets of \(A_1,\dots, A_m\).
    We claim that \(|\bigcup_{i\in[m]}A_i|\le 3|F_0|\). 
    Note that there is at most one set, denoted by \(A_{i_0}\), not in $F_0$ or $F_1$ -- if two such sets remain, say \(A_{i_0}\) and \(A_{i_1}\), then by the maximality of $F_0$ and $F_1$, we have \(|F_0\cup A_{i_0}|\ge n/2C\) and \(|F_1\cup A_{i_1}|\ge n/2C\), which implies there is an edge between them, a contradiction. 
    Since \(| A_{i_0}|\le|F_1|\le|F_0|\), it follows that \(|\bigcup_{i\in[m]}A_i|\le 3|F_0|\).
    Therefore, by the expansion property, we have 
    \[
    m\le|\cup_{i\in[m]}A_i|\le 3|F_0|\le |N_G(F_0)|\le |S|. \qedhere
    \]
\end{proof}
Next, we prove the lower bound on the threshold for the appearance of a perfect matching (and hence of a Hamilton cycle).
\begin{proposition}\label{pro:lower_bound_of_threshold}
     Let $\gamma\in(0,1]$ and let \(1/n\ll\eps\ll\alpha\le 1\). If $G$ be an \(n\)-vertex $(q,\beta)$-bijumbled graph with $\beta\le \eps qn$, $\delta(G)\ge \alpha qn$ and $qn=\Omega(\log n)$, then for \(p(n)= \frac{\log n}{ 4qn}\), $G_p$ contains no perfect matching with high probability. In particular, if \(G\) is an \((n,d,\lambda)\)-graph with \(\lambda\le \eps d\), then this holds already for \(p(n)=(1-\gamma)\frac{\log n}{d}\).
\end{proposition}
\begin{proof}
Suppose $p=\frac{\log n}{ 4qn} < 1$.
To show that $G_p$ has no perfect matchings, it suffices to show that there is an isolated vertex with high probability. For $v\in V(G)$, let $I_v$ be the indicator that $v$ is an isolated vertex in $G_p$ and $I=\sum_{v\in V(G)}I_v$ be the number of isolated vertices in $G_p$.
Let $A:=\{v\in V(G):d(v)\le 2qn\}$ and $B:=V(G)\setminus A$, then we have 
\[
2qn|B|\le e_G(V(G), B)\le qn |B| + \beta n,
\]
implying $|B| \le \beta/q\le \varepsilon n$. 
Therefore, we have
\[\mathbb{E}(I)=\sum_{v\in V(G)}\mathbb{P}[I_v]\ge\sum_{v\in A}\mathbb{P}[I_v]\ge (1-\varepsilon)n(1-p)^{2qn}=\Omega\left(n^{1/2}\right).
\]
On the other hand, as $p nq=\log n/4$ and $(1-p)^{\alpha nq}\le e^{-\alpha pnq}\le n^{-\alpha/4}$, we have 
    \begin{align*}
        \text{Var}(I)&=\sum_{v,u\in V(G)}\l(\mathbb{E}(I_v I_u)-\mathbb{E}(I_v)\mathbb{E}(I_u)\r)\le \sum_{u v\in E(G)}p(1-p)^{d(v)+d(u)-1}\\
        &\le qn^2p(1-p)^{2\alpha nq}\le n^{1-\alpha/2}\log n=o(\mathbb{E}^2(I)).
    \end{align*} 
By the second moment method, $I>0$ almost surely, which implies that there is no perfect matchings with high probability. In particular, for \((n,d,\lambda)\)-graphs we can further show that when \(p=(1-\gamma)\frac{\log n}{d}<1\), \(G_p\) has isolated vertices with high probability.
\end{proof}

Now, we are ready to prove Theorem~\ref{thm:main_thm_bijumbled}. 

\begin{proof}[Proof of Theorem~\ref{thm:main_thm_bijumbled}]
The lower bound follows directly from Proposition~\ref{pro:lower_bound_of_threshold}, so it remains to prove the upper bound. 
Let $\gamma>0$ be an arbitrary constant and \(p=(1+\gamma)\frac{\log n}{\alpha qn}\). 
By ~\Cref{thm:robust_expander} and Lemma~\ref{lem:tutte–expander_lemma}, $G_p$ contains a perfect matching with high probability; the same holds for a Hamiltonian cycle by Theorem~\ref{expander graph Hamiltonian}.  
\end{proof}

Moreover, \Cref{thm:main_thm_ndlambda} can be derived from \Cref{thm:main_thm_bijumbled} with parameters $q=\frac{d}{n},~\beta=\lambda$ and \(\alpha=1\), using Proposition~\ref{pro:lower_bound_of_threshold}.

\section{Sparse graph coupling}\label{sparse couple}
In this section, we will prove a coupling lemma for triangles in an expander graph $G$ and hyperedges in the $K_3$-hypergraph of $G$ (see definition below). 
This coupling result will help us to get the correct order of magnitude of $p$ in~\Cref{thm:main_thm_triangle_factor}. 
We must note that both the statement and the proof of the lemma rely heavily on the work of Riordan~\cite{riordan2022random}, which corresponds to the case when $G$ is a complete graph.

Before stating the lemma, we introduce some structures we will use in the proof. 
The main idea of such a coupling result is to compare the triangles in a random graph and the edges of a random 3-uniform hypergraph. For this, we define the clique-hypergraph as follows.
Given a graph $G$, let $H$ be the $K_r$-\emph{hypergraph} of $G$, which is a $r$-uniform hypergraph such that $V(H)=V(G)$ and $E(H)$ consists of all copies of $K_r$ in $G$.

Our proof of the coupling lemma is similar to the approach used by Riordan in \cite{riordan2022random}. Roughly speaking, we test for the presence of each possible triangle in the sparse graph $G$ one by one with probability $p^3$.  
Therefore, it suffices to show that, at least on a global event of high probability, the conditional probability that a particular test succeeds given the history is at least $p^3$. 
If this conditional probability exceeds $p^3$, we need to ``thin'' it. 
Suppose this conditional probability is $p'>p$, we then toss a coin with head probability $p^3 / p'$ so that the resulting conditional probability becomes exactly $p^3$. 
In addition, to control the failure probability, we define certain bad events with occurrence probability $o(1)$, such as the appearance of some dense configurations or vertices with low degree. We will show that if these bad events do not occur, the coupling succeeds, and hence our coupling lemma holds.


In our setting, the underlying graph 
$G$ is already sparse. Therefore, the definition of the bad event requires more careful treatment than in the dense case. To control the distribution of triangles in each test, we introduce a family of forbidden configurations generated by $5$ linear $3$-cycles with one common edge. This part differs from that used in \cite{riordan2022random}. The configurations in this family are designed so that they do not occur in random hypergraphs of the density we consider.

We define a family of forbidden configurations as follows.
Let $W$ be the graph with $V(W)=\{v_i,u_i~|i\in [3]\}$ and $E(W)=\{v_1v_2,v_2v_3,v_3v_1, v_1u_1,v_2u_1,v_2u_2, v_3u_2,v_1u_3,v_3u_3\}$ and let $C_3^{(3)}$ be 
the $K_3$-hypergraph of $W$ with vertex set $V(W)$ and edge set $\{v_1u_1v_2,v_2u_2v_3,v_3u_3v_1\}$ (i.e. $C_3^{(3)}$ is the $3$-uniform linear $3$-cycle). Then let $\mathcal{F}$ be the collection of graphs formed by $5$ linear $3$-cycles $C_1, C_2, C_3, C_4, C_5$ with common edge $v_1u_1v_2$ such that the other edges are distinct. Note that we do not require that these $5$ linear $3$-cycles be vertex-disjoint in addition to $v_1,u_1,v_2$. Then we have the following counting result for $\mathcal{F}$.

\begin{proposition}\label{count}
        Let $0\le 1/n\ll \varepsilon, 1/d\le1$ and $\lambda\le \varepsilon d$. Let $G$ be an $(n,d,\lambda)$-graph with $\lambda \le \varepsilon d$ and $H$ be the $K_3$-hypergraph of $G$.
        Then there are at most $2^6\frac{d^{23}}{n^5}$ copies of graphs in $\mathcal{F}$.
\end{proposition}
\begin{proof}
Suppose that those $5$ linear $3$-cycles are $C_i$ for $i\in [5]$.  By Lemma~\ref{expandermixing}, there are at most $d^3$ choices for the common hyperedge $v_1u_1v_2$ (i.e., corresponding to a triangle in $G$) in these $5$ linear 3-cycles. Fix a tuple $v_1u_1v_2$. For a cycle $C_i$ with $i\in [5]$, suppose that $V(C_i)=\{v_i,u_i~|i\in [3]\}$ and $E(C_i)=\{v_1u_1v_2,v_2u_2v_3,v_3u_3v_1\}$. Since $u_2v_3\in E(G[N(v_2)])$ and $e_G(N_G(v_2))\le \frac{2d^3}{n}$, there are at most $\frac{2d^3}{n}$ choices for $u_2,v_3$. Meanwhile, $u_3$ is a neighbor of $v_1$ and $d_G(v_1)=d$. Hence, there are at most $d$ choices for $u_3$. Thus, there are at most $\frac{2d^3}{n}\cdot d=\frac{2d^4}{n}$ choices for left vertices of $C_i$ after setting $v_1,u_1,v_2$.
Since we have $5$ linear $3$-cycles, there are at most \[d^3\cdot \l(\frac{2d^4}{n}\r)^5=\frac{2^5d^{23}}{n^5}\] copies of graphs in $\mathcal{F}$. 
\end{proof}

Now, we state our sparse graph coupling lemma. 
\begin{lemma}[Coupling Lemma]\label{sparse couplelemma}
    Let $0< 1/n\le1/d\ll\varepsilon, ,C\le1$ and suppose that $\lambda\le \varepsilon d$. Let $G$ be an $(n,d,\lambda)$-graph with $d\gg n^{2/3}\log n$, and let $H$ be the $K_3$-hypergraph of $G$. Then the following holds for any $p=p(n,d)\le C(n\log n)^{\frac{1}{3}} d^{-1}$. Let $a<\frac{1}{2^{10}}$ be a constant, and let $\pi=\pi(n,d)=ap^3$. Then $G\cap G(n,p)$ can be coupled with the random hypergraph $H\cap H_3(n,\pi)$ so that, with high probability, for every hyperedge in $H\cap H_3(n,\pi)$ there is a copy of $K_3$ in $G\cap G(n,p)$ with the same vertex set. 
\end{lemma}
\begin{proof}
    Fix a constant $0<c<1$ such that the following holds:
    \[c(1-2^{8}c)>a.\]
    Let $G$ be an $n$-vertex $(n,d,\lambda)$-graph and let $E_1,\dots, E_m$ be the edge sets of all triangles \(F_1,\dots,F_m\) in $G$.
    In addition to the random variables corresponding to the edges of \(G\cap G(n,p)\), we consider an indicator variable \(I_j\) for each triangle \(F_j\) in \(G\), with \(\mathbb{P}[I_j=1]=c\). 
    Note that \(I_j\) is independent of the presence of the edges in \(G\cap G(n,p)\), and that it refers to a different quantity from the random variable describing the presence of the corresponding hyperedge in \(H\cap H_3(n,\pi)\).
    We will construct a random hypergraph \(H'\) by using \(I_i\) to `thin' the triangles in \(G\cap G(n,p)\) so that \(H'\) has the same distribution as $H\cap H_3(n,\pi)$.
     
    Now, we consider the random (non-uniform) hypergraph \(G^*\) with vertex set \(V(G)\) and edge set \[E(G\cap G(n,p))\cup \{h_i:I_i=1\},\] where each \(h_i\) is the \(3\)-edge induced by the triangle \(F_i\). 
    Define $A_i$ as the event that \(E_i^*:=E_i\cup \{h_i\}\subseteq E(G^*)\).
    Now we construct $H^{\prime}$ by the following algorithm, revealing information of \(G^*\) while simultaneously constructing \(H'\) step by step. 

\textbf{Algorithm:} For each $j$ from $1$ to $m$:

\textbf{Setup}: Calculate the conditional probability \(\pi_j\) of the event \(A_j\) given all information revealed so far.

If $\pi_j\ge \pi$, then flip a coin with head probability $\pi/\pi _j$. If it lands heads, then we check whether $A_j$ holds. If so, we declare that the hyperedge $h_j$ corresponding to $F_j$ is present in $H'$; otherwise, we exclude it.

If $\pi_j< \pi $, then flip a coin with head probability \(\pi\), and include \(h_j\) in \(H'\) if this coin lands head. If this happens, our coupling has failed.

\textbf{Output:} a random hypergraph $H^{\prime}$.

Observe that the hypergraph \(H'\) constructed according to the above algorithm follows the same distribution as $H\cap H_3(n,\pi)$.
It thus remains to check that the coupling fails with probability \(o(1)\). 
Indeed, if this procedure succeeds, then we embed \(H'\) within the ``thinned'' triangle hypergraph, which has a hyperedge for each triangle \(F_j\) in \(G\cap G(n,p)\) with \(I_j=1\), thereby yielding the conclusion of the theorem.
To this end, it suffices to show that the probability that $\pi_j<\pi$ and the hyperedge corresponding
to $E_j$ is present in $H'$ is $o(1)$.

Suppose we have reached step \(j\). Next, we will estimate \(\pi_j\). 
Note that in the previous step, we checked whether certain (not necessarily all) events \(A_i\) hold, and in each case  we received the answer ``yes'' or ``no''.
Let \(Y\) be the (random) set of indices \(i\) for which the events \(A_i\) hold, and let \(X\) denote the set of indices for which events \(A_i\) do not hold. 
Let \(R=\cup_{i\in Y}E_i^*\) be the set of (hyper)edges of \(G^*\) found so far.
For $i< j$, let $E_i':=E_i^*\backslash R$; if \(i\in X\), then \(\{h_j\}\notin R \), which implies that \(E_i'=(E_i\setminus R)\cup \{h_i\}\).
Then what we know about \(G^*\) is that all (hyper)edges in \(R\) are present, and for every \(i\in X\), not all (hyper)edges in \(E_i'\) are present. 
Let $G'$ be the random hypergraph on $V(G)$ where all edges in \(R\) are automatically included, each \(2\)-edge not in \(R\) is included independently with probability  $p$, and each \(3\)-edge \(h_i\) not in 
\(R\) is included independently with probability \(c\).
For each \(i<j\), let $A_i'$ be the event that $E_i'\subseteq E(G')$. Then, we have  

\[\pi_j=\mathbb{P}[A_j'~|~\bigcap_{i\in X}(A_i')^c]=\mathbb{P}[E_j'\subseteq E(G')~|~\bigcap_{i\in X}\{E_i'\nsubseteq E(G')\}].\]

Next, we define three bad events. Let $B_1$ be the event that there is a vertex in at least $2\log n$ hyperedges of $H\cap H_3(n,\pi)$. Let $B_2$ be the event that there is a copy of $\mathcal{F}$ in $H\cap H_3(n,\pi)$. Let $B_3$ be the event that there is a copy of $F'$ with $V(F')=\{w_1,\dots, w_5\}$ and $E(F')=\{w_1w_2w_3,w_3w_4w_5,w_2w_4w_5 \}$ in $H\cap H_3(n,\pi)$.
We will show that if $\pi_j<\pi$ and the hyperedge corresponding
to $E_j$ is present in $H'$ (the only case where the coupling fails), then $B_1\cup B_2\cup B_3$ occurs. 
\begin{claim}\label{clm:bad_event}
    $\mathbb{P}[B_1\cup B_2\cup B_3]=o(1)$ in $H\cap H_3(n,\pi)$. 
\end{claim}
\begin{proof}[Proof of claim]
    Note that for a vertex $v\in V(G)$, $d_G(v)=d$. By the expander mixing lemma, 
    \[e(N_G(v))\le \frac{1}{2}\left(\frac{d}{n}d^2+\lambda d\right)\le \frac{d^3}{n}.\]
    Hence, $v$ is in at most $\frac{d^3}{n}$ triangles in $G$ and so in at most $\frac{d^3}{n}$ hyperedges of $H\cap H_3(n,\pi)$. 
    Since $\pi=ap^3\le\frac{cn\log n}{d^3}$, by Lemma~\ref{lem:chernoff_bound}, with probability $o(1)$, there exists a vertex in at least $2\log n$ hyperedges of $H$.

    Note that by definition, each graph in $\mathcal{F}$ contains exactly $11$ edges. 
    Let $X$ be the number of copies of graphs in $\mathcal{F}$. Then by~Proposition~\ref{count} and $d\ge n^{2/3}$, we see that 
    \begin{align*}
        \mathbb{E}[X]\le \pi^{11}\cdot \frac{2^{6}d^{23}}{n^5}\le a^{11}\frac{2^6n^6}{d^{10}}\log^{11} n=o(1).
    \end{align*}
    By the first moment method, with probability $o(1)$, there exists a copy of some graph in $\mathcal{F}$. 

For the event $B_3$, let $Y$ be the number of copies of $F'$. By Lemma~\ref{expandermixing}, there are at most $n\cdot \frac{d^3}{n}=d^3$ hyperegdes in $H$. Hence, there are at most $d^3$ possible choices for the $w_1w_2w_3$. Note that $w_4w_5$ is an edge in $N_G(w_3)$ and $e(N_G(w_3))\le \frac{d^3}{n}$. Thus, there are at most $2\frac{d^6}{n}$ copies of $F'$ in $H$ and so
    $\mathbb{E}(Y)\le 2\frac{d^6}{n} \cdot \pi^3=o(1).$
By the first moment method, with probability $o(1)$, there exists a copy of $F'$. Thus, we have $\mathbb{P}[B_1\cup B_2\cup B_3]=o(1)$. \qedhere

\end{proof}
Let $N_1$ be the set of $i\in X$ such that $E_i'\cap E_j'\neq \emptyset$, and let $N_2:=X\setminus N_1$.  
Then 
\begin{align*}
    \pi_j&=\mathbb{P}[A_j'~|~\bigcap_{i\in N_1}(A_i')^c\cap \bigcap_{i\in N_2}(A_i')^c]\ge \mathbb{P}[A_j'\cap \bigcap_{i\in N_1}(A_i')^c~|~\bigcap_{i\in N_2}(A_i')^c]\\
    &=\mathbb{P}[A_j'~|~\bigcap_{i\in N_2}(A_i')^c]-\mathbb{P}[A_j'\cap \bigcup_{i\in N_1}A_i'~|~~\bigcap_{i\in N_2}(A_i')^c]\\
    &\ge\mathbb{P}[A_j']-\mathbb{P}[A_j'\cap \bigcup_{i\in N_1}A_i']
\ge \mathbb{P}[A_j']-\sum_{i\in N_1}\mathbb{P}[A_j'\cap A_i']\\
&= cp^{|E_j\setminus R|}-\sum_{i\in N_1}c^2p^{|(E_i\cup E_j)\backslash R|}
    =p^{|E_j\backslash R|}\left(c-c^2\sum_{i\in N_1}p^{|E_i\backslash \l(E_j\cup R\r)|}\right).
\end{align*}
The second inequality holds since $A_j'$ and $\bigcap_{i\in N_2}(A_i')^c$ are independent and $A_{j}'\cap \bigcup_{i\in N_1}A_i'$ is an increasing event, while $\bigcap_{i\in N_2}(A_i')^c$ is a decreasing event. 
The third inequality holds due to the union bound. 
In the second-to-last equation, the factor \(c\) in the first term comes from the probability that \(h_j\) is in \(G'\), while the factor \(c^2\) in the second term comes from the probability that both \(h_j\) and \(h_i\) are in \(G'\).

Let $S=\sum_{i\in N_1}p^{|E_i\backslash \l(E_j\cup R\r)|}$ and $N_{1\alpha}=\{i~|~i\in N_1 \text{ and 
 }|E_i\backslash \l(E_j\cup R\r)|=\alpha\}$ for $\alpha\in \{0,1,2\}$. Then $S=\sum_{\alpha=0}^{2}\sum_{i\in N_{1\alpha}}p^\alpha$.
 Note that to prove $\pi_j\ge\pi$ with high probability, it suffices to show that $S<2^{8}$ with high probability.
 
Since $E_i'\cap E_j'\neq \emptyset$ for every $i\in N_1$,  we have \(|E_i\cap E_j|=1\). 
Hence, $N_{1i}\le \sum_{uv\in E_j}d_{G}(uv)\le 3d$ for $i\in \{0,1,2\}$. 
This implies that \[\sum_{i\in N_{12}}p^{2}\le 3d\cdot C^2\cdot  \frac{n^{\frac{2}{3}}}{d^2} \log n^{\frac{2}{3}}=o(1).\]
Moreover, if $|E_i\backslash \l(E_j\cup R\r)|=1$, then we have exactly two edges of $E_i$ that are in $E_j\cup R$, and the common vertex of these two edges is one vertex of $E_j$ (since $E_i'\cap E_j'\neq \emptyset$). 
By Claim~\ref{clm:bad_event}, with high probability, every vertex in \(E_j\) belongs to at most $2\log n$ triangles in $R$. 
Hence, there are at most $3\cdot 2\log n\cdot 4$ choices of such $E_i$. Thus, with high probability, \[\sum_{i\in N_{11}}p\le 24\log n\cdot C\frac{n^{1/3}}{d}\log^{1/3} n=o(1).\] 
Now, the remaining part is to prove that $|N_{10}|< 2^{8}$ with high probability. 
It suffices to show that if $|N_{10}|\ge 2^{8}$, then $B_2$ happens in $H\cap H_3(n,\pi)$. 
Let $i_0\in N_{10}$. Then $E_{i_0}\subseteq E_j\cup R$ and $E_{i_0}\cap E_j\neq \emptyset$. 
This implies that there are distinct $E_{j_1},E_{j_2}\subseteq R\backslash E_j$ with $E_{j_i}\cap E_{i_{0}}\neq\emptyset$ for $i\in [2]$. Note that $E_{j_1}\cap E_{j_2}=\emptyset$, otherwise $E_{j_1},E_{j_2}$ and $E_{j}$ forms a copy of $F'$ and $B_3$ occurs. Thus $E_{j_1},E_{j_2},\text{ and }E_{j}$ form a linear $3$-cycle in $H^{\prime}$. Let $E_j=e_1e_2e_3$. Let $N_{10}^s$ be the set of $i\in N_{10}$ with $E_i\cap E_j=e_s$ for $s\in [3]$. 
Suppose to the contrary that \[\sum_{i\in N_{10}}1=|N_{10}|=|N_{10}^1|+|N_{10}^2|+|N_{10}^3|\ge 2^{8}.\]
Then there exists \(s\in [3]\) such that $|N_{10}^s|\ge 2^{6}$. 
Without loss of generality, assume that~$|N_{10}^1|\ge 2^{6}$. Let $E_{i_1},\dots,E_{i_r}$ be the triangles with $i_s\in N_{10}^1$ for $s\in [r]$. 
By the discussion above, denote the linear $3$-cycle determined by $E_{i_s}$ as $C_{i_s}$ for $s\in [r]$. Note that these $C_{i_s}$ have a common $3$-edge $E_j$.
 
Next, we claim that there exists a copy of some graph in $\mathcal{F}$ formed by members of $C_{i_1}, \dots, C_{i_r}$. 
Indeed, let $F$ be a maximal collection of $C_{i_s}$ such that any two of them intersect exactly in the same \(3\)-edge \(E_j\). 
By the definition of~$\mathcal{F}$, if $|F|\ge 5$, then we are done. 
Suppose instead $|F|\le 4$. 
This implies that each cycle $C_{i_s}\notin F$ intersects some cycle in $F$ with at least two $3$-edges. 
Moreover, this means that all vertices of $E_{i_s}$ are contained in hyperedges of the cycle of~$F$. 
However, note that there are at most $6\cdot4=24$ vertices in the cycles of $F$. 
Hence, there are at most $\binom{24}{1}\le 2^{5}$ choices for $E_{i_s}$, implying that $|N^1_{10}|\le 2^5+|F|=36 < 2^6$, a contradiction.
Therefore, if $|N_{10}|>2^{8}$, then $B_2$ happens.
The proof is completed.
\end{proof}

\section{Robust triangle factor}\label{iterative ab}
\subsection{Spreadness}
As mentioned in the introduction, deriving Theorem~\ref{thm:main_thm_triangle_factor} from a spreadness function requires us to pass through the recent breakthrough result of Frankston, Kahn, Narayanan, and Park~\cite{FKNP}. 
\begin{defn}
    Let $q\in [0,1]$. Let $(V,\mathcal{H})$ be a hypergraph, and let $\mu$ be a probability distribution on $\mathcal{H}$. We say that $\mu$ is \emph{$q$-spread} if
    \[\mu(\{A\in \mathcal{H}~:~S\subseteq A\})\le q^{|S|}\text{ for all $S\subseteq V$}.\]
\end{defn}
In our context, we are primarily concerned with such a hypergraph \(\mathcal{H}\) where \(V:=E(H)\) and \(H\) is the $K_3$-hypergraph of $G$, and \(\mathcal{H}\) denotes the collection of perfect matchings of \(H\).

Frankston, Kahn, Narayanan, and Park (FKNP) proved the following theorem. 
Given a hypergraph \(H\), we say that \(H\) is \(r\)-uniform if every edge has size exactly \(r\).
\begin{theorem}[\cite{FKNP}]\label{FKNP}
    If $(V,\mathcal{H})$ is an $r$-uniform hypergraph and $\mathcal{H}$ supports a $q$-spread distribution, then there exists an absolute constant $K$ such that a $p$-random subset of $V$ contains an edge in $\mathcal{H}$ a.a.s if $p\ge Kq\log r$ as $r\rightarrow \infty$.
\end{theorem}
Now we are ready to outline the proof of Theorem~\ref{thm:main_thm_triangle_factor}. 
By Theorem~\ref{FKNP} and Lemma~\ref{sparse couplelemma}, it suffices to show that there exists an $O(n/d^3)$-spread measure on the set of perfect matchings in the \(K_3\)-hypergraph of \(G\). To construct such a measure, we use the iterative absorption method, which consists of three steps. We first select a chain of induced subgraphs of the host graph (Lemma \ref{lem:vortex_lemma}), specifically by randomly choosing $G_N\subseteq\dots \subseteq G_1\subseteq G$. Since each induced subgraph \(G_i\) has a suitable pseudorandom property, we can build the desired measure in each $G_i$ successively. This will be accomplished by randomly covering almost all the vertices in $G_i$ by itself and dealing with all remaining vertices by $G_{i+1}$ (Lemma~\ref{lem:cover_down_lemma}). Finally, the leftover vertices in $G_N$ are covered by Lemma~\ref{thm:bijumbled_graph}.

For the rest of this section, after some preparation, we prove the vortex lemma (Lemma~\ref{lem:vortex_lemma}) in Section~\ref{sec:vortex_lemma} and the cover-down lemma (Lemma~\ref{lem:cover_down_lemma}) in Section~\ref{sec:cover_down_lemma}, and finally prove Theorem~\ref{thm:main_thm_triangle_factor} in Section~\ref{sec:triangle_factor}.

\subsection{Find an almost triangle factor}

Given a graph \(G\), we denote $H_3(G)$ by the \(K_3\)-hypergraph of \(G\). 
The following lemma gives us a spread probability distribution on the set of almost perfect matchings in $H_3(G)$ in a sparse setting.

\begin{lemma}{}\label{lem:find_almost_traingle_factors_in_bijumbled_graphs}
Let $0<1/n\ll\varepsilon\ll
\eta\le 1$, $n^{-2/3}\ll q\le1$ and $\beta\le\varepsilon q^2n$. 
Then there exists a constant $C_1=C_1(\eta)$ such that the following holds for all sufficiently large $n$. 
Assume $G$ is a $(q,\beta)$-bijumbled graph on the vertex set $[n]$. 
Then there exists a $(\frac{C_1}{q^3n^2})$-spread probability distribution $\mathcal{D}$ on the set of matchings in $H_3(G)$ that cover at least $(1-\eta)n$ vertices.

\end{lemma}
\begin{proof}{}{}
Let $0<1/n\ll\varepsilon\ll
\eta\le 1$.
Since $G$ is a $(q,\beta)$-bijumbled graph on the vertex set $[n]$, for any $A, B\subset[n]$, we have 
\[
e(A,B)=q|A||B|\pm\beta\sqrt{|A||B|}.
\]
Set $t=\lceil(1-\eta)n/3\rceil$.
Now, we will construct a random disjoint triangle tuple $(X_1,X_2,\ldots,X_t)$ that covers at least $(1-\eta)n$ vertices of $G$ as follows.

Assume that we have already found triangles $X_1,\ldots, X_{i-1}$, and let $V_i=V(G)\setminus{\cup_{j=1}^{i-1}V(X_j)}$.
Let $G_i=G[V_i]$ and $A_i=\{v\in V_i:d_{G_i}(v)\ge\frac{1}{2} q|V_i|\}$. By the bijumbledness, we have 
\[
e_{G_i}(V_i,V_i\setminus{A_i})\ge q|V_i||V_i\setminus{A_i}|-\beta\sqrt{|V_i||V_i\setminus{A_i}|}.
\]
Moreover, $e_{G_i}(V_i,V_i\setminus{A_i})=\sum_{v\in V_i\setminus{A_i}}d_{G_i}(v)\le\frac{1}{2}q|V_i||V_i\setminus{A_i}|$. Then 
\[
q|V_i||V_i\setminus{A_i}|-\beta\sqrt{|V_i||V_i\setminus{A_i}|}\le \frac{1}{2}q|V_i||V_i\setminus{A_i}|,
\]
which implies that $|A_i|\ge|V_i|/2\ge \eta n/2$. 
First, we sample a vertex $v_i\in A_i$ uniformly at random. 
Since $d_{G_i}(v_i)\ge q|V_i|/2$ and $\varepsilon\ll\eta$, there are at least 
\begin{align*}
\frac{1}{2}e\left(G_i\left[N_{G_i}(v_i)\right]\right)&\ge \frac{1}{2}\left(q\left|N_{G_i}(v_i)\right|^2-\beta\left|N_{G_i}(v_i)\right|\right)=\frac{1}{2}\left|N_{G_i}(v_i)\right|\left(q\left|N_{G_i}(v_i)\right|-\beta\right)\\
&\ge\frac{1}{4}q\eta n\left(\frac{1}{2}q^2\eta n-\varepsilon q^2n\right)\ge\frac{1}{9}q^3\eta^2n^2
\end{align*}
triangles which are incident to $v_i$ in $G_i$, then we randomly sample such a triangle to be $X_i$.

Let $(X_1,\ldots, X_t)$ be the resulting random triangle tuple and define a random matching $M_1=\{X_1,\ldots, X_t\}$. 
Obviously, $M_1$ covers at least $(1-\eta)n$ vertices in $G$.
Now, we start to prove that $M_1$ is an $O(\frac{1}{q^3n^2})$-spread matching in $H$.
For $r\in[t]$, let $R=\{T_1,\ldots,T_r\}$ be a set of $r$ disjoint triangles in $G$. 
Let $\pi\in S_r$ be an arbitrary permutation on $[r]$. Then 
\begin{align*}
\mathbb{P}\left[R\subset M_1\right]
&\le \sum_{\pi\in S_r}\sum_{1\le i_1<\cdots<i_r\le t}
\mathbb{P}\left[X_{i_k}=T_{\pi(k)},~\forall k\in[r]\right]\\
&\le r!\binom{t}{r}\max_{\overset{1\le i_1<\ldots<i_r\le t,}{\pi\in S_r}}\mathbb{P}\left[X_{i_k}=T_{\pi(k)},~\forall k\in[r]\right]\\
&\le r!\binom{t}{r}\left(\frac{3}{\frac{1}{2}\eta n\frac{1}{9}q^3\eta^2n^2}\right)^r\le \left(\frac{18}{q^3\eta^3n^2}\right)^r.\qedhere
\end{align*}

\end{proof}

\subsection{Vortex lemma}\label{sec:vortex_lemma}
In this section, we prove the following lemma, which guarantees a \emph{distribution} over vortices rather than the existence of any specific one. A crucial feature of our analysis is that the randomness in the choice of the vortex is taken into account when calculating the spread.

\begin{lemma}\label{lem:vortex_lemma}
  Let $d,n$ be positive integers, \(0<1/n\ll 1/C\ll\eps\ll\alpha, 1/C_{\ref{thm:random_subgraphs_of_spectral_expanders}}\le 1\), and $\gamma\in(C_{\ref{thm:random_subgraphs_of_spectral_expanders}}n^{-\frac{1}{6}}\log^{\frac{1}{2}}n, \frac{1}{100})$. Assume that \(G\) is an \((n,d,\lambda)\)-graph with $\lambda n^{1/6}\log^{-1/2}{n}\ge d\gg n^{1/3}$, then there is a distribution on the set of sequences \(V(G)=V_0\supseteq V_1\supseteq\cdots\supseteq V_N=X\), where \(N\le \log_{1/\alpha}n\), with the following properties: 
\begin{enumerate}[label=\rm{(\roman*)}]
        \item\label{item:1.2.1} For every \(0\le i<N\), we have \(|V_{i+1}|=\lceil\alpha^2|V_i|\rceil\);
        \item\label{item:1.2.2} \(|V_N|\in\left[\alpha^2 n^{4/3}/d, n^{4/3}/d\right]\);
        \item\label{item:1.2.3} For every \(v\in V(G)\), every \(0\le i\le N\), we have \(d(v,V_i)=(1\pm\gamma)p_id\), where \(p_i=|V_i|/n\);
        \item\label{item:1.2.4} For every \(0\le i\le N\), \(G[V_i]\) is an \((|V_i|,(1\pm2\gamma)p_id,6p_i\lambda)\)-graph;
        \item\label{item:1.2.5} For every vertex set \(\{v_1,\ldots,v_m\}\subseteq V(G)\) and every vector \(\vec{x}\in [N]^m\), we have 
        \[\mathbb{P}\left[\bigwedge_{i=1}^m(v_i\in V_{x_i})\right]\le{\prod_{i=1}^m\frac{2|V_{x_i}|}{n}}.\]
        
\end{enumerate}

\end{lemma}
\begin{proof}
First, consider the distribution on the set of sequences \(V(G)=U_0\supseteq U_1\supseteq\cdots\supseteq{U_N}=X\) obtained as follows: Set $U_0=V(G)$. 
For as long as $|U_i|>n^{4/3}/d$, let $U_{i+1}$ be a uniformly random subset of $U_i$ of size exactly $\lceil\alpha^2|U_i|\rceil$. 

Let $E$ be the event that properties~\ref{item:1.2.1} to~\ref{item:1.2.4} hold.
Observe that properties~\ref{item:1.2.1} and ~\ref{item:1.2.2} hold by definition. 
Lemma~\ref{lem:chernoff_bound}, Theorem~\ref{thm:random_subgraphs_of_spectral_expanders} and a union bound imply that properties~\ref{item:1.2.3} and~\ref{item:1.2.4} hold with high probability. 
Note that for every nonempty $\{v_1,\ldots,v_m\}\subseteq V(G)$ and evevry $\vec{x}\in\{0,\cdots,N\}^m$, \[\mathbb{P}\left[\bigwedge_{i=1}^m(v_i\in U_{x_i})\right]\le{\prod_{i=1}^m\frac{|U_{x_i}|}{n}},\] which follows from the hypergeometric distribution. 

Let $V_0\supseteq V_1\supseteq\cdots\supseteq V_N$ be the distribution obtained by conditioning $V(G)=U_0\supseteq U_1\supseteq\cdots\supseteq{U_N}$ on the occurrence of $E$. 
By definition, $V_0\supseteq V_1\supseteq\cdots\supseteq V_N$ satisfies properties~\ref{item:1.2.1} to~\ref{item:1.2.4}. 
Furthermore, for every nonempty $\{v_1,\ldots,v_m\}\subseteq V(G)$ and $\vec{x}\in\{0,\cdots,N\}^m$, we have
\[\mathbb{P}\left[\bigwedge_{i=1}^m(v_i\in V_{x_i})\Big|E\right]\le \frac{\mathbb{P}\left[\bigwedge_{i=1}^m(v_i\in V_{x_i})\right]}{\mathbb{P}[E]}\le{\prod_{i=1}^m\frac{2|V_{x_i}|}{n}}. \qedhere\]
\end{proof}

\subsection{Cover-down lemma}\label{sec:cover_down_lemma}
Let $G=(V, E)$ be an almost regular expander and $U\subseteq V$ be a small subset.
The following cover-down lemma states that it can find a spread distribution on matchings that cover all vertices in $V\setminus U$ by using at most an $\eps$-fraction of the vertices of $U$. Here, we use a novel algorithm to complete this step.

\begin{lemma}\label{lem:cover_down_lemma}
Let \(0<1/n\ll\eps\ll\alpha,c\le 1\), \( n^{-2/3}\ll q\le 1\), and \(\beta\le \eps q^2n\)
then there exists a constant \(C_1=C_1(\alpha,c)>0\) such that the following holds for all sufficiently large \(n\). 
Let \(G\) be a \(\l(q,\beta\r)\)-bijumbled graph on the vertex set $[n]$, and let \(U\subset V(G)\) be a subset of size \(\lceil\alpha^2n\rceil\). Suppose that \(G\) satisfies for all $v\in[n]$, 
\(d_G(v,U)\ge cq\card{U}\).
Then there exists a \(\l(C_1/q^3n^2\r)\)-spread probability distribution on the set of matchings $M\subseteq{H_3(G)}$ that satisfies: 
\begin{enumerate}[label=\rm{(\roman*)}]
        \item\label{item:1.3.1} \(M\) covers every vertex in \(V(G)\setminus{U}\);
        \item\label{item:1.3.2} \(M\) covers at most \(\alpha^2\card{U}\) vertices in \(U\);
\end{enumerate}

\end{lemma}
\begin{proof}
We begin by splitting \(U\) into two parts, each assigned to cover a different group of vertices in \(V(G)\).  
Let \(\pi\) be a permutation on \([|U|]\) that is chosen uniformly at random. Define
\begin{align*}
U_1:= \left\{\pi(i)~:~i\in\left[1,~\lceil\alpha^2|U|\rceil\right]\right\},
U_2 := \left\{\pi(i)~:~i\in\left[\lceil\alpha^2|U|\rceil+1,~|U|\right]\right\}.
\end{align*}
Let \(E\) be the event that for every \(v\in V(G)\) and \(j\in [2]\), \(d_G(v,U_j)\ge(1-\eps)cq\card{U_j}\). 
We will condition on \(E\), which holds with probability at least \(99/100\) by Lemma~\ref{lem:chernoff_bound}.

Let \(V':=V(G)\setminus U\) and \(G':=G[V']\).
Note that \(G'\) is the subgraph of \(G\) induced on \(V'\), and thus \(G'\) is \(\l(q,\beta\r)\)-bijumbled. 
We apply Lemma~\ref{lem:find_almost_traingle_factors_in_bijumbled_graphs} to \(G'\) to find an \(O_{\alpha}\l(1/q^3n^2\r)\)-spread random matching \(M_1\subseteq H_3(G')\) that covers all but at most \(\alpha^{7}\card{V'}\le \alpha^4|U|\) vertices. 
Next, we cover all the remaining vertices, denoted by \(W=V'\setminus V(M_1)\), using vertices from \(U\).
 


Conditioning on $M_1$, we randomly sample a set $M_2$ of triangles to cover all remaining vertices in $W$ as follows.
Given an enumeration $v_1, v_2,\ldots,v_m$ of the vertices in $W$, assume we have already covered $v_1,\ldots,v_{i-1}$ by triangles $T_1,\ldots,T_{i-1}$ respectively. 
Let $W_{i-1}=\{v_i,\ldots,v_m\}$ and $U_{j,i-1}=U_j\setminus{\cup_{k=1}^{i-1}V(T_k)}$ for \(j\in[2]\), where $U_{j,0}:=U_j$. 
Call the vertex $v_i$ is \emph{bad} if $d_G(v_i,U_{1,i-1})<\alpha^{5}qn$. 
For each $i\in\left[m\right]$, let \(B_{i-1}\) be the set of bad vertices \(v_j\) with \(j<i\).

If $v_i$ is not ``bad'', i.e., $d(v_i,U_{1,i-1})\ge\alpha^{5}qn$, we cover it by using \(U_1\). 
Recall that \(G\) is a \(\l(q,\beta\r)\)-bijumbled graph and \(\eps\ll \alpha\), then there are at least
\begin{align*}
\frac{1}{2}e\l(N_{U_{1,i-1}}\l(v_i\r)\r)\ge \frac{1}{2}\left|N_{U_{1,i-1}}\l(v_i\r)\right|\left(q\left|N_{U_{1,i-1}}\l(v_i\r)\right|-\beta\right)\ge 
\frac{1}{2}\alpha^{5}qn\left(\alpha^{5}q^2n-\eps q^2n\right)\ge\frac{1}{4}{\alpha}^{10}q^3n^2
\end{align*}
triangles (with exactly two vertices in $U_{1,i-1}$) that are incident to $v_i$, and we randomly sample such a triangle.

If $v_i$ is ``bad'', we will cover $v_i$ using the vertices of $U_2$. 
Note that $e\left(B_{i-1}, U_{1,i-1}\right)< \alpha^{5}qn\left|B_{i-1}\right|$. 
On the other hand, \(e\left(B_{i-1}, U_{1,i-1}\right)\ge q\left|B_{i-1}\right|\left|U_{1,i-1}\right|-\beta\sqrt{\left|B_{i-1}\right|\left|U_{1,i-1}\right|}.\)
Therefore, we have 
\begin{align*}
\alpha^{5}qn\left|B_{i-1}\right|
\ge q\left|B_{i-1}\right|\left|U_{1,i-1}\right|-\beta\sqrt{\left|B_{i-1}\right|\left|U_{1,i-1}\right|}&\ge q\left|B_{i-1}\right|\left(\alpha^2-2\alpha^4\right)|U|-\eps q^2 n\sqrt{\left|B_{i-1}\right|\alpha^2|U|}\\
&\ge q\left|B_{i-1}\right|\left(\alpha^4/2\right)n-\alpha^2\eps q^2 n\sqrt{\left|B_{i-1}\right|n}.
\end{align*}
Then, by \(\eps\ll \alpha, c\), for all $i\le m$, $\left|B_{i-1}\right|\le 16 \alpha^{-4}\eps^2 q^2 n\le {\alpha}^6 cqn$, which implies that the number of vertices in $U_2$ covered by the previous triangles is at most $2\left|B_{i-1}\right|\le 2{\alpha}^6 cqn$.
Thus, for each $v\in V(G)$ we have 
\begin{align*}
    d_G(v, U_{2,i-1})&\ge d_G(v, U_2)-2\left|B_{i-1}\right|
    \ge(1-\eps)cq\card{U_2}-2{\alpha}^6 cqn\\
    &\ge (1-\eps)cq{\alpha}^2(1-\alpha^4)n-2\alpha^6cqn\ge \alpha^2cqn/2.
\end{align*}  
Therefore, for each ``bad'' vertex $v_i$, there are at least 
\[\frac{1}{2}e\l(N_{U_{2,i-1}}\l(v_i\r)\r)\ge 
\frac{1}{4}\alpha^{2}cqn\left(\frac{1}{2}\alpha^{2}cq^2n-\eps q^2n\right)\ge\frac{1}{10}{\alpha}^{4}c^2q^3n^2\]
available triangles for $v_i$, regardless of the previously made selections, and we randomly sample such a triangle $T_i$.
Thus, we have $M_2=\{T_1,\ldots,T_m\}$. 
Set $M:=M_1\cup M_2$. 
Then $M$ is a \(3\)-uniform matching that covers all vertices in $V\setminus U$ by using at most $2\alpha^4|U|$ vertices in $U$.

It remains to show that $M$ is $O\l(1/q^3n^2\r)$-spread. 
Let $S\subseteq H_3(G)$ be a set of hyperedges. 
We need to show that $\mathbb{P}[S\subseteq M]=\left(O\l(1/q^3n^2\r)\right)^{|S|}$. 
First, we assume that $S$ is a matching. Let \(S=S_1\cup S_2\), where $S_1$ is those hyperedges in $S$ with all vertices in $V\setminus U$, and $S_2:=S\setminus S_1$. 
We now have 
\[\mathbb{P}\left[S\subseteq M\right]=\mathbb{P}\left[S_1\subseteq M_1\right]\mathbb{P}\left[S_2\subseteq M_2~|~S_1\subseteq M_1\right].\]

By~\Cref{lem:find_almost_traingle_factors_in_bijumbled_graphs}, $M_1$ is $O\l(1/q^3n^2\r)$-spread, so $\mathbb{P}\left[S_1\subseteq M_1\right]=\left(O(1/q^3n^2)\right)^{|S_1|}$. 
Next, we observe that after conditioning on any outcome of $M_1$, it holds that $S_2\subseteq M_2$ only if for every hyperedge $e\in S_2$, the hyperedge chosen to match the vertex in $e\setminus U$ is $e$. 
Since every such choice is made uniformly at random from at least $\min\{{\alpha}^{10}q^3n^2/4,{\alpha}^{4}c^2q^3n^2/10\}$ possibilities regardless of the previous selections, it follows that 
\[\mathbb{P}\left[S_2\subseteq M_2~|~S_1\subseteq M_1\right]=\left(O_{\alpha,c}(1/q^3n^2)\right)^{|S_2|}.\]
Thus, $\mathbb{P}\left[S\subseteq M\right]=\left(O_{\alpha,c}(1/q^3n^2)\right)^{|S|}$, as desired.
\end{proof}

\subsection{Proof of Theorem \ref{thm:main_thm_triangle_factor}}\label{sec:triangle_factor}
Now, we are ready to show our theorem for the existence of triangle factors in $G_p$.

\begin{proof}[Proof of Theorem~\ref{thm:main_thm_triangle_factor}]
Let $0<1/n\ll1/C\ll\varepsilon\ll\alpha\ll1/C_{\ref{thm:random_subgraphs_of_spectral_expanders}}\le1$ and $\gamma=\max\{C_{\ref{thm:random_subgraphs_of_spectral_expanders}}n^{-1/6}\log^{1/2}n, \lambda/d\}$. As $\lambda\le\gamma d$ and $G$ is an $(n,d,\lambda)$-graph, then \(G\) is also an \((n,d,\gamma d)\)-graph. 
Moreover, when $\gamma=C_{\ref{thm:random_subgraphs_of_spectral_expanders}}n^{-1/6}\log^{1/2}n$, since $d\ge Cn^{5/6}\log^{1/2}n$ and $1/C\ll\varepsilon\ll1/C_{\ref{thm:random_subgraphs_of_spectral_expanders}}$, it follows that $\gamma d\le\varepsilon d^2/n$; and when $\gamma d=\lambda$, it is clear that $\gamma d\le\varepsilon d^2/n$.
Now, let \(H_3(G)\) the \(K_3\)-hypergraph of \(G\). Let $\mathcal{M}$ be a perfect matching on $n$-vertex $3$-graph. 

To prove \Cref{thm:main_thm_triangle_factor}~\ref{traingle factor1}, it suffices, by Theorem~\ref{FKNP} and Lemma~\ref{sparse couplelemma}, to show that there exists an $O(n/d^3)$-spread distribution on copies of~$\mathcal{M}$ in $H_3(G)$. 
Applying Lemma~\ref{lem:vortex_lemma} to the $(n,d,\gamma d)$-graph $G$, one obtain a random sequence of sets \(V(G)=V_0\supseteq V_1\supseteq\dots\supseteq V_N\) that satisfies the properties~\ref{item:1.2.1} to \ref{item:1.2.5} in Lemma~\ref{lem:vortex_lemma}. 
Note that the parameter conditions required by Lemma~\ref{lem:vortex_lemma} hold for our choice of parameters, which can be verified by a straightforward computation. 

We will inductively construct a (random) sequence of matchings \(\emptyset=M_0\subseteq M_1\subseteq \dots\subseteq M_N\) in \(H_3(G)\), satisfying the following properties for every \(1\le i\le N\). For notational convenience, set \(V_{N+1}=\emptyset \).

\begin{enumerate}[label=(A\arabic*)]
    \item \(M_{i}\setminus M_{i-1}\) is \(O\left(n^3/\card{V_{i-1}}^2d^3\right)\)-spread;
    \item \(M_{i}\) covers all vertices in \(V(H)\setminus V_{i}\);
    \item \(|V(M_{i})\cap V_{i}|\le \alpha^2 |V_{i}|\);
    \item \(V(M_{i})\cap V_{i+1}=\emptyset\).
\end{enumerate}
    
We begin by taking \(M_0=\emptyset\). Now suppose that for \(1\le i\le N\), we have constructed \(M_i\) with the properties above. 
Let \(V_i'=V_i\setminus (V(M_i)\cup V_{i+2})\), \(G_i=G[V_i']\), and \(U_i=V_{i+1}\setminus V_{i+2}\).

On the one hand, by the property \ref{item:1.2.4} in Lemma~\ref{lem:vortex_lemma}, we have \(G[V_i]\) is an \((|V_i|,(1\pm2\gamma)p_id,6p_i\gamma d)\)-graph, where \(p_i=\card{V_i}/n\). 
Lemma~\ref{lem:mixing_lemma_for_almost_expnaders} implies that for any two subsets \(S,T\subseteq V_i'\), we have 
\begin{align*}
    e\left(S,T\right)
    \ge{\frac{(1-2\gamma)^2d\card{S}\card{T}}{(1+2\gamma)n}-\frac{1+2\gamma}{1-2\gamma}6p_i\gamma d\sqrt{\card{S}\card{T}}}&\ge \frac{(1-6\gamma)d\card{S}\card{T}}{n}-6(1+4\gamma)p_i\gamma d\sqrt{\card{S}\card{T}}\\
    &\ge\frac{d}{n}\card{S}\card{T}-6(2+4\gamma)p_i\gamma d\sqrt{\card{S}\card{T}}
\end{align*}
and \(e\left(S,T\right)\le \frac{d}{n}\card{S}\card{T}+6(2+4\gamma)p_i\gamma d\sqrt{\card{S}\card{T}}\). 
Thus, \(G_i\) is \(\l(d/n, 6(2+4\gamma)p_i\gamma d\r)\)-bijumbled.

On the other hand, for every \(v\in V_i'\), by Lemma~\ref{lem:vortex_lemma}\ref{item:1.2.3} it holds
\[d_G(v,U_i)=d_G\l(v,V_{i+1}\r)-d_G\l(v,V_{i+2}\r)=\l(1\pm3\alpha\r)\frac{\card{U_i}}{n}d.\]
Applying  Lemma \ref{lem:cover_down_lemma} to $G_i$ with \(U_i\), and setting \(q=d/n,~\beta=6(2+4\gamma)p_i\gamma d,~c=1-3\alpha\), 
we obtain an \(O\left(n^3/\card{V_i}^2d^3\right)\)-spread matching \(M_i'\) covering all vertices in \(V_i'\setminus V_{i+1}\) and at most \(\alpha^2 |V_{i+1}|\) vertices in \(V_{i+1}\), and no vertex in \(V_{i+2}\). 
By taking \(M_{i+1}=M_i\cup M_i' \) we complete the inductive step.

Finally, to obtain a perfect matching, note that if \(M_N\) satisfies the properties above, then $\delta(G[V_N\setminus V(M_N)])\ge\left(1-2\alpha\right)\frac{|V_N\setminus{V(M_N)}|}{n}d$. Moreover, \(G[V_N\setminus V(M_N)]\) is \(\l(d/n, 6(2+4\gamma)p_N\gamma d\r)\)-bijumbled and \(\frac{\alpha^2n^{4/3}}{2d}\le |V_N\setminus V(M_N)|\le \frac{n^{4/3}}{d}\). 
As $\gamma d\le\varepsilon d^2/n$, by applying Theorem \ref{thm:bijumbled_graph} with \(c=1-2\alpha,~q=d/n,~\beta=6(2+4\gamma)p_N\gamma d\), and \(\eta=\frac{15\eps}{\alpha^4}\), 
we obtain a triangle factor \(\widetilde{F}\subseteq G[V_N\setminus V(M_N)]\), which corresponds to a matching \(\widetilde{M}\) in \(H_3(G)\). Take \(M=M_N\cup \widetilde{M}\).

It remains to prove that $ M $ is \(O\left(n/d^3\right)\)-spread. 
Let $ S\subseteq E(H_3(G)) $ be a matching. 
We need to show that \(P_S=\mathbb{P}[S\subseteq M]=(O(n/d^3))^{|S|}\). 
Let $ T_1,\dots, T_m$ be an enumeration of the hyperedges in $S$. 
For each vector $\overrightarrow{x}\in [N+1]^m$, let $ P(\overrightarrow{x}) $ be the probability that for every $ j\in[m] $, the hyperedge $ T_j $ is in $M_{x_j}\setminus M_{x_{j}-1}$ if $ x_j\le N $, and $ T_j\in \widetilde{M} $ if $ x_j=N+1 $. 
We will show that 
\begin{equation}\label{prob}
P\left(\overrightarrow{x}\right)= \left(\prod_{i=1}^N{\left(O_{\alpha}\left(\frac{\card{V_{i-1}}}{d^3}\right)\right)}^{|\{j:x_j=i\}|}\right){\left(\frac{|V_N|}{n}\right)}^{3|\{j:x_j=N+1\}|} .
\end{equation}
This will suffice, since then 
\begin{equation}
\begin{aligned}
P_S=\sum_{\overrightarrow{x}\in[N+1]^m}P\left(\overrightarrow{x}\right)	&=\sum_{\overrightarrow{x}\in[N+1]^m}\left(\prod_{i=1}^N{\left(O_{\alpha}\left(\frac{\card{V_{i-1}}}{d^3}\right)\right)}^{|\{j:x_j=i\}|}\right){\left(\frac{|V_N|}{n}\right)}^{3|\{j:x_j=N+1\}|}\\
&={\left(O_{\alpha}\left(\frac{n}{d^3}\right)\right)}^m\sum_{\overrightarrow{x}\in[N+1]^m}\left(\prod_{i=1}^N{\left(\alpha^{2i}\right)}^{|\{j:x_j=i\}|}\right)\left(\frac{d^3}{n^4}\cdot\frac{n^{4}}{d^3}\right)^{|\{j:x_j=N+1\}|}\\
&={\left(O_{\alpha}\left(\frac{n}{d^3}\right)\right)}^m\prod_{j=1}^m\left(\sum_{i=1}^{N}\alpha^{2i}+1\right)={\left(O_{\alpha}\left(\frac{n}{d^3}\right)\right)}^m.
\end{aligned}
\end{equation}
We now prove (\ref{prob}). For \(1\le i\le N\), let \(C_i\) be the event that \(\{T_j:x_j=i\}\subseteq G[V_{i-1}]\). 
Let $ D_i $ be the event that $ \{T_j:x_j=i\}\subseteq M_i\setminus M_{i-1}$ if $i\le N $, $ \{T_j:x_j=i\}\subseteq \widetilde{M}$ if $ i=N+1 $.  
We then have that 
\begin{equation}
P\left(\overrightarrow{x}\right)\le\mathbb{P}[\bigcap_{i=1}^{N+1}C_i]\prod_{i=1}^{N+1}\mathbb{P}[D_i|\bigcap_{i=1}^{N+1}C_i,D_1\cap\dots \cap D_{i-1}].
\end{equation}
By the randomness guaranteed in the vortex construction (\ref{item:1.2.5} in Lemma \ref{lem:vortex_lemma}), we have 
\begin{equation}
\mathbb{P}\left[\bigcap_{i=1}^{N+1}C_i\right]=\left(\prod_{i=1}^{N+1}{\left(O\left(\frac{|V_{i-1}|}{n}\right)\right)}^{3|\{j:x_j=i\}|}\right).
\end{equation}
Note that conditioned on any outcome of $ M_{i-1} $, the matching $ M_i\setminus M_{i-1} $ is $ O\l(n^3/\card{V_{i-1}}^2d^3\r) $-spread. 
Thus, for every $ i\le N $:
\[\mathbb{P}\left[D_i|\bigcap_{i=1}^{N+1}C_i,D_1^0\cap\dots\cap D_{i-1}^0\right]=\left(O_{\alpha}\left(n^3/\card{V_{i-1}}^2d^3\right)\right)^{|\{j:x_j=i\}|}.\]
Finally, we use the trivial bound $ \mathbb{P}\left[D_{N+1}|\bigcap_{i=1}^{N+1}C_i,D_1\cap\dots\cap D_{N+1}\right]\le 1 $ to obtain:
\[P\left(\overrightarrow{x}\right)\le\left(\prod_{i=1}^{N+1}{\left(O\left(\frac{|V_{i-1}|}{n}\right)\right)}^{3|\{j:x_j=i\}|}\right)\left(\prod_{i=1}^N\left(O_{\alpha}\left(\frac{n^3}{\card{V_{i-1}}^2d^3}\right)\right)^{|\{j:x_j=i\}|}\right)\]
which implies (\ref{prob}).

On the other hand, to prove~\Cref{thm:main_thm_triangle_factor}~\ref{traingle factor2}, it suffices to prove that if $p=\min\{(1-\eps)n^{1/3}\log^{1/3} n/{d}, 1\}$, then with high probability there exists a vertex which is not covered by a triangle in $G\cap G(n,p)$. 
For $v\in V(G)$, let $I_v$ be the indicator that $v$ is not covered by a triangle in $G\cap G(n,p)$. 
Recall that $G$ is an $(n,d,\lambda)$-graph. By Lemma~\ref{expandermixing}, $e_G(N(v))=\frac{d^3}{2n}\pm \lambda d$. 
Since $\lambda\le \eps d^2/n$, we get that $e_G(N(v))=(1\pm \eps)d^3/(2n)$. 
Hence, each vertex $v$ is in $(1\pm \eps)d^3/(2n)$ triangles in $G$.

Let $I=\sum_{v\in V(G)}I_v$ be the number of vertices that are not covered by a triangle in $G\cap G(n,p)$. 
For each vertex $v\in V(G)$, let $\{v,y,z\}$ be a triangle in $G$, and let $X_{v,y,z}$ be the event that the triangle $\{v,y,z\}$ is in $G\cap G(n,p)$.
Notice that $I_v=\bigwedge_{y,z}\bar{X}_{v,y,z}$, then by Harris' inequality, we have 
\[
\mathbb{P}[I_v]\ge\prod_{y,z}\mathbb{P}\left[\bar{X}_{v,y,z}\right]\ge{\left(1-p^3\right)}^{\frac{(1+\eps)d^3}{2n}}=\omega(1/n).
\]
Hence, \[\mathbb{E}(I)= \sum_{v\in V(G)}\mathbb{P}[I_v]\ge\omega(1).\]
By the first moment method, $I>0$ almost surely. 
This implies that there is no triangle factors in $G$ with high probability.
\end{proof}

\section{Concluding remark}~\label{concluding}
In this paper, we consider the random subgraphs of pseudorandom graphs. Note that our method for the triangle factor in pseudorandom graphs can be generalized to all clique factors.
We remark that Theorems~\ref{thm:main_thm_ndlambda} and~\ref{thm:main_thm_triangle_factor} enable us to count the number of Hamiltonian cycles, perfect matchings, and triangle factors in the pseudorandom graphs.
\begin{coro}
    Let $G$ be an $(n,d,\lambda)$-graph with $\lambda=o(d)$ and sufficient large $n$. Then $G$ contains at least $(\tfrac{d}{(1+o(1))\log n})^n$ Hamiltonian cycles and $(\frac{d}{(1+o(1))\log n})^{\frac{n}{2}}$ perfect matchings. 
\end{coro}

\begin{coro}
    Let $0<1/n\ll\varepsilon\ll1$. There exists a constant $c>0$ such that if $G$ is an $(n,d,\lambda)$-graph with $d=\Omega(n^{5/6}\log^{{1}/{2}}n)$, $3|n$ and $\lambda\le \frac{\varepsilon d^2}{n}$, then $G$ contains at least $(\frac{d}{cn^{{1}/{3}}\log^{{1}/{3}} n})^n$ triangle factors.
\end{coro}

It would be interesting to figure out what the weakest possible requirement on the degree $d$ and the spectral gap is that will guarantee the existence of triangle factors in the random subgraphs. 

\begin{question}
What is the smallest $d$ such that for any $(n,d,\lambda)$-graph $G$ with $\lambda=o(\tfrac{d^2}{n})$ and $p\gg \tfrac{n^{\frac{1}{3}}\log^{\frac{1}{3}} n}{d}$, $G_p$ contains a triangle factor with high probability?
\end{question}


\bibliographystyle{plain}
\bibliography{cite}

\end{document}